\documentclass{amsart}
\usepackage{amsthm}
\usepackage{amsmath}
\usepackage{amssymb}
\usepackage{mathrsfs}
\usepackage{enumerate}
\usepackage{float}
\usepackage{url, multicol}
\usepackage[all]{xy}
\usepackage{graphicx}
\usepackage{hyperref}

\usepackage{url}

\newcommand{\Q}{{\mathbb Q}}

\newcommand{\Z}{{\mathbb Z}}
\newcommand{\F}{{\mathbb F}}
\newcommand{\N}{{\mathbb N}}

\newcommand{\OO}{{\mathcal O}}

\newcommand{\p}{{\mathfrak p}}
\newcommand{\<}{\langle}
\renewcommand{\>}{\rangle}

\newcommand{\rank}{\mathrm{rank}}

\newcommand{\Gal}{\mathrm{Gal}}
\renewcommand{\ker}{\mathrm{ker}}

\newcommand{\coker}{\mathrm{coker}}

\newcommand{\Hom}{\mathrm{Hom}}
\DeclareMathOperator{\ord}{ord}

\title[{Generalizations of Iwasawa's `Riemann-Hurwitz' Formula}]{GENERALIZATIONS OF IWASAWA'S `RIEMANN-HURWITZ' FORMULA FOR CYCLIC $p$-EXTENSIONS OF NUMBER FIELDS}

\author{JORDAN SCHETTLER}

\begin{document}
\newtheorem{thm}{Theorem}
\newtheorem{conj}[thm]{Conjecture}
\newtheorem{prop}[thm]{Proposition}
\newtheorem{lemma}[thm]{Lemma}
\newtheorem{corollary}[thm]{Corollary}
\theoremstyle{remark}
\newtheorem{rem}[thm]{Remark}
\theoremstyle{definition}
\newtheorem{defn}[thm]{Definition}
\newtheorem{exam}[thm]{Example}
\numberwithin{equation}{section}
\numberwithin{equation}{thm}

\begin{abstract}
We produce generalizations of Iwasawa's `Riemann-Hurwitz' formula for number fields. These generalizations apply to cyclic extensions of number fields of degree $p^n$ for any positive integer $n$. We first deduce some congruences and inequalities and then use these formulas to establish a vanishing criterion for Iwasawa $\lambda$-invariants which generalizes a result of Takashi Fukuda et. al. for totally real number fields.
\end{abstract}

\maketitle


%\ccode{Mathematics Subject Classification 2010: 11R23, 11R80}

\section{Introduction}

Fix a rational prime $p$ and an algebraic closure $\overline{\Q}$ of $\Q$. Let $\Q_{\infty} \subseteq \overline{\Q}$ denote the unique $\Z_p$-extension of $\Q$.
%, i.e., the Galois group $\Gal(\Q_{\infty}/\Q)$ is isomorphic to the $p$-adic integers $\Z_p$.
In particular, we have
\begin{align*}
\Q_{\infty} \subseteq \bigcup_{n \geq 1} \Q(\zeta_{p^n})
\end{align*}
where $\zeta_{p^n}$ denotes a primitive $p^n$th root of unity. Following Iwasawa in \cite{Iwas}, we define a $\Z_p$-field to be a finite extension of $\Q_{\infty}$. Equivalently, $L$ is a $\Z_p$-field when $L = {l}\Q_{\infty}$ for some number field ${l}$, so here $L$ is the cyclotomic $\Z_p$-extension of ${l}$. We define the ideal class group of a $\Z_p$-field $L$ to be the quotient $C_L \mathrel{\mathop:}=  I_L/P_L$ where $I_L$ is the group of invertible fractional ideals of the ring of integers $\OO_L$ in $L$ and where $P_L \leq I_L$ is the subgroup of principal fractional ideals.
\begin{thm}[Iwasawa]
Let $L$ be a $\Z_p$-field and let $A_L$ denote the $p$-primary part of the class group of $L$. Then there is an isomorphism of $\Z_p$-modules
\begin{align*}
A_L \cong (\Q_p/\Z_p)^{\lambda_L} \oplus M
\end{align*}
where $M$ has bounded exponent, i.e., $p^nM = 0$ for some $n$. In fact, if we write $L = {l}\Q_{\infty}$ for some number field ${l}$, then the Iwasawa invariants $\lambda(L/{l})$, $\mu(L/{l})$ for the $\Z_p$-extension $L/{l}$ satisfy
\begin{enumerate}
\item $\lambda(L/{l}) =\lambda_L $
\item $\mu(L/{l}) = 0 \Leftrightarrow M = 0$.
\end{enumerate}
In particular, this means that the vanishing of $\mu(L/{l})$, as conjectured by Iwasawa, only depends on $L$, so we may write $\mu_L=0$ to denote this.
\end{thm}
\begin{proof}
See \cite{Iwas5} and \cite{Iwas4}.
\end{proof}
In \cite{Iwas}, Iwasawa used the above structure thm and Galois actions on class groups to prove the following.
\begin{thm}[Iwasawa's `Riemann-Hurwitz' Formula]\label{tm}
Suppose $L/K$ is a cyclic extension of $\Z_p$-fields of degree $[L:K] = p$. If $L/K$ is unramified at the infinite places and $\mu_K=0$, then $\mu_L=0$ and
\begin{align}\label{Iwasform}
\lambda_L = [L:K]\lambda_K + (p-1)(h_2-h_1) + \sum_{w \nmid p} (e(w)-1)
\end{align}
where $e(w)$ denotes the ramification index in $L/K$ of a place $w$ of $L$ not lying above $p$ and for $i=1,2$ we write $h_i$ for the $\F_p$-dimension of the cohomology group $H^i(\Gal(L/K), \OO_L^{\times})$.
\end{thm}

\section{The Euler Characteristic}\label{facts}

We wish now to restate Iwasawa's formula (\ref{Iwasform}) in a way which will lend itself more conveniently to generalization.
\begin{defn}
Let $G$ be a cyclic group of prime power order $p^n$. Suppose $M$ is a $G$-module. We define the Euler characteristic $\chi(G,M) \in \Z$ to be the exponent of $p$ in the Herbrand quotient
\begin{align*}
p^{\chi(G, M)} = \frac{|H^2(G, M)|}{|H^1(G,M)|}
\end{align*}
when these quantities are finite.
\end{defn}
Note that $\chi$ inherits the following properties (see \cite{Arti}) directly from the Herbrand quotient:
\begin{enumerate}
\item $\chi$ is additive on short exact sequences of $G$-modules%\footnote{If $0\rightarrow A \rightarrow B \rightarrow C \rightarrow 0$ is a SES of $G$-modules and two of the numbers $\chi(G, A)$, $\chi(G, B)$, $\chi(G, C)$ are finite (i.e., well-defined), then so is the other and $\chi(G, A) -\chi(G, B) + \chi(G, C)= 0$.}
\item $\chi(G,M) = 0$ when $M$ is a finite $G$-module
\item $\chi(G, M^{\ast}) = -\chi(G,M)$ when $M^{\ast} = \Hom_{\Z_p}(M, \Q_p/\Z_p)$ is the $p$-Pontryagin dual of a $\Z_pG$-module $M$.
\end{enumerate}
These properties and the techniques of \cite{Iwas} can be used to derive the following computations.
\begin{lemma}\label{chi}
Suppose $L/K$ is a cyclic $p$-extension of $\Z_p$-fields with $G = \Gal(L/K)$. Then
\begin{align*}\chi(G,A_L) = -\chi(G, P_L)  + \chi(G,I_L)\end{align*}
and
\begin{align*}\chi(G,I_L) = \ord_p|I_L^G/I_K| = \sum_{u \nmid p} \ord_p(e(w/u))\end{align*}
where $\ord_p(e(w/u))$ is the $p$-adic order of the ramification index in $L/K$ for a finite place $w$ of $L$ lying over a place $u$ of $K$ which does not lie over $p$. If, in addition, $L/K$ is unramified at every infinite place, then
\begin{align*} -\chi(G,P_L) = \chi(G, \OO_L^{\times}).\end{align*}
\end{lemma}
\begin{corollary}
We can restate Iwasawa's formula (\ref{Iwasform}) as
\begin{align}\label{above1}
\lambda_L = p\lambda_K + (p-1)\chi(G, A_L)
\end{align}
In fact, we need not assume that $L/K$ is unramified at the infinite places for Eq. (\ref{above1}) above to hold.
\end{corollary}

\section{General Formulas}

We will derive several generalizations of Iwsawa's formula, but first we need a couple of lemmas. The first lemma is inspired by the ideas of Heller and Reiner in \cite{Hell}.
\begin{lemma}\label{Rei}
Let $G = \langle g \rangle \cong \Z/(p^n)$ for some prime $p$ and some positive integer $n$. Suppose $M$ is a $\Z_pG$-module which is free of finite rank over $\Z_p$. Then there is a short exact sequence of $\Z_pG$-modules
\begin{align*}0 \rightarrow M' \rightarrow M \rightarrow \Z_p[\zeta_{p^n}]^{\oplus r} \rightarrow 0\end{align*}
where $M'$ is a $\Z_p$-pure\footnote{Recall that if $M$ is an $R$-module ($R$ a commutative ring with $1$), we say a submodule $N\leq M$ is $R$-pure when $rM \cap N \subseteq rN$ for every $r\in R$.} $\Z_pG$-submodule of $M$ which is annihilated by $g^{p^{n-1}}-1$ and $\Z_p[\zeta_{p^n}]$ has $\Z_pG$-module structure given by
\begin{align*}\Z_p[\zeta_{p^n}] \cong \frac{\Z_pG}{\Phi_{p^n}(g)\Z_pG}\end{align*}
with $\Phi_{p^n}(x) = p^n$th cyclotomic polynomial.
\end{lemma}
\begin{proof}
Define
\begin{align*}M':= \{m\in M : (g^{p^{n-1}}-1)m=0\}.\end{align*}
Then $M'$ is a $\Z_pG$-submodule of $M$ since it is the kernel of a $\Z_pG$-homomorphism, namely, the multiplication by $g^{p^{n-1}}-1$ map on $M$. We know $M'$ is $\Z_p$-pure since if $rm=m'$ where $r\in \Z_p$, $m\in M$, and $m' \in M'$, then
\begin{align*}r((g^{p^{n-1}}-1)m) =(g^{p^{n-1}}-1)(rm) = (g^{p^{n-1}}-1)m' = 0, \end{align*}
so $(g^{p^{n-1}}-1)m=0$ (i.e., $m\in M'$) because $M$ is $\Z_p$-torsion free. Also, $M/M'$ is annihilated by $\Phi_{p^n}(g)$ since
\begin{align*}(g^{p^{n-1}}-1)(\Phi_{p^n}(g)m) = ((g^{p^{n-1}}-1)(\Phi_{p^n}(g))m = (g^{p^n}-1)m = 0 \end{align*}
for all $m \in M$. Thus $M/M'$ is a $\Z_p[\zeta_{p^n}]$-module which (since $M' \leq M$ is $\Z_p$-pure and $\Z_p$ is a PID) is free of finite rank over $\Z_p$. Note that $\Z_p \cap \Z_p[\zeta_{p^n}]\alpha$ is a non-zero ideal of $\Z_p$ when $0\neq \alpha \in \Z_p[\zeta_{p^n}]$. Thus if $\alpha \overline{m} = 0$ for some $\overline{m} \in M/M'$, then $r\overline{m} = \beta(\alpha\overline{m}) = 0$ where $0\neq r= \beta\alpha \in \Z_p$ for some $\beta \in \Z_p[\zeta_{p^n}]$. This implies $\overline{m} = 0$ because $M/M'$ is $\Z_p$-free. Hence $M/M'$ is torsion free as a $\Z_p[\zeta_{p^n}]$-module; moreover, $M/M'$ is finitely generated over $\Z_p[\zeta_{p^n}]$ since it is finitely generated over $\Z_p$. Therefore $M/M'$ is free of finite rank over $\Z_p[\zeta_{p^n}]$ since $\Z_p[\zeta_{p^n}]$ is a PID.
\end{proof}
\begin{lemma}\label{work}
Let $G = \langle g \rangle \cong \Z/(p^n)$ for some prime $p$ and some nonnegative integer $n$. Suppose $M$ is a $\Z_pG$-module which is free of finite rank over $\Z_p$. Then there is a sequence $r_0, \ldots, r_n$ of nonnegative integers such that for every subgroup $N_i = \langle g^{p^i} \rangle$ with $0\leq i \leq n$ we have
\begin{align*}\rank_{\Z_p}(M^{N_i}) = \sum_{t=0}^{i} r_t \varphi(p^t)  \end{align*}
and
\begin{align*}\chi(N_i, M) = (n-i)\sum_{t=0}^i r_t \varphi(p^t) - p^i \sum_{t=i+1}^n r_t. \end{align*}
\end{lemma}
\begin{proof}
We use induction on $n$ and Lemma \ref{Rei}. If $n=0$, then the proposition is clear with $\rank_{\Z_p}(M^{N_0})=\rank_{\Z_p}(M) = r_0$. Now suppose $n \geq 1$ and the proposition is true for $n-1$. By Lemma \ref{Rei}, we have a short exact sequence of $\Z_pG$-modules
\begin{align*} 0 \rightarrow M' \rightarrow M \rightarrow \Z_p[\zeta_{p^n}]^{\oplus r_n} \rightarrow 0 \end{align*}
where $M'$ can be regarded as a $\Z_pG'$-module where $G' = G/N_{n-1} \cong \Z/(p^{n-1})$. By induction, there is a sequence $r_0, \ldots, r_{n-1}$ of nonnegative integers such that for every subgroup $N'_i = N_i/N_{n-1} \leq G'$ with $0 \leq i \leq n-1$ we have
\begin{align*}\rank_{\Z_p}(M'^{N'_i}) =  \sum_{t=0}^{i} r_t \varphi(p^t) \end{align*}
and
\begin{align*} \chi(N'_i, M') = (n-1-i)\sum_{t=0}^i r_t \varphi(p^t) - p^i \sum_{t=i+1}^{n-1} r_t \end{align*}
since $\Z_p[\zeta_{p^n}]^{N_{i}}=0$. We need to compute the difference $\chi(N_i, M') - \chi(N'_i, M')$. Since $(g^{p^i})^{p^{n-i-1}} = g^{p^{n-1}}$ acts as $1$ on $M'$, we know that $1+g^{p^i}+\cdots +(g^{p^i})^{p^{n-i}-1}$ acts as $p(1+g^{p^i}+\cdots +(g^{p^i})^{p^{n-i-1}-1})$ on $M'$. Thus $H^1(N_i,M') = H^1(N_i',M')$ since $M'$ is $\Z_p$-free, and the multiplication by $p$ map $H^2(N'_i, M') \rightarrow H^2(N_i, M')$ is an injection with cokernel $M'^{N_i}/pM'^{N_i}$. Hence
\begin{align*}
\chi(N_i, M') - \chi(N'_i, M')&= |M'^{N_i}/pM'^{N_i}|=\rank_{\Z_p}(M'^{N_i}) \\
&= \rank_{\Z_p}(M'^{N'_i}) =  \sum_{t=0}^{i} r_t \varphi(p^t) ,\end{align*}
so
\begin{align*}
\chi(N_i, M) &= \chi(N_i, M') + r_n\chi(N_i, \Z_p[\zeta_{p^n}]) \\
&= \chi(N'_i, M') + \sum_{t=0}^{i} r_t \varphi(p^t) +  r_n\chi(N_i, \Z_p[\zeta_{p^n}]) \\
&= (n-i)\sum_{t=0}^i r_t \varphi(p^t) - p^i \sum_{t=i+1}^{n-1} r_t + r_n\chi(N_i, \Z_p[\zeta_{p^n}]),
\end{align*}
but $H^2(N_i, \Z_p[\zeta_{p^n}]) = 0$ and
\begin{align*}H^1(N_i, \Z_p[\zeta_{p^n}]) = \frac{\Z_p[\zeta_{p^n}]}{(\zeta_{p^n}^{p^i} - 1)} \cong \frac{\Z_p[x]}{(x^{p^i} -1)+(\Phi_{p^n}(x))} \cong \frac{\Z_p[\Z/(p^i)]}{(\Phi_{p^n}(1))} = \frac{\Z_p[\Z/(p^i)]}{(p)},  \end{align*}
so $\chi(N_i, \Z_p[\zeta_{p^n}]) = -p^i$ as needed. Also, it is clear that $\chi(N_n, M) = 0$ and
\begin{align*}
\rank_{\Z_p}(M^{N_n}) &= \rank_{\Z_p}(M) \\
&= \rank_{\Z_p}(M') + r_n\rank_{\Z_p}(\Z[\zeta_{p^n}]) \\
&= \sum_{t=0}^{n-1} r_t \varphi(p^t) + r_n\varphi(p^n),
\end{align*}
which finishes the proof.
\end{proof}
Now we compute the $n+1$ unknown $r_i$'s in Lemma \ref{work} in terms of $n+2$ arithmetic invariants consisting of Iwasawa lambda invariants and Euler characteristics of class groups.
\begin{thm}\label{above}
Let $p$ be prime and $K_0 \subseteq K_1 \subseteq \ldots \subseteq K_n$ be a tower of $\Z_p$-fields such that $n\geq 1$ and for all $i$ the extension $K_i/K_0$ is cyclic of degree $p^i$. Suppose $\mu_{K_0} = 0$. Then $\mu_{K_1} = \cdots  = \mu_{K_n} = 0$ and
\begin{align*} \sum_{i=0}^{n-1} \varphi(p^i)\lambda_{K_{n-i}}  = p^{n-1}(np-n+1)\lambda_{K_0} + \varphi(p^n)\chi(G_n, A_{K_n})  \end{align*}
where $G_n = \Gal(K_n/K_0)$.
\end{thm}
\begin{proof}
Theorem \ref{tm} implies $\mu_{K_1} = \cdots  = \mu_{K_n} = 0$ by induction. We apply Lemma \ref{work} to the $\Z_pG_n$-module $A_{K_n}^{\ast}$ (the $p$-Pontryagin dual of the $p$-primary part of the class group), which is free of finite rank $\lambda_{K_n}$ over $\Z_p$. Thus there is a sequence of nonnegative integers $r_0, r_1, \ldots, r_n$ such that for all $i = 0, 1, \ldots, n$ we have
\begin{align*}
\lambda_{K_i} &= \rank_{\Z_p}(A_{K_i}^{\ast}) = \rank_{\Z_p}((A_{K_n}^{\ast})^{N_i}) = \sum_{t=0}^i r_t \varphi(p^t) \\
\chi(G_n, A_{K_n}) &= -\chi(N_0, A_{K_n}^{\ast}) = -nr_0 + \sum_{t=1}^n r_t
\end{align*}
where $N_i = \Gal(K_n/K_i)$. Note that the natural map $C_{K_i} \rightarrow C_{K_n}^{N_i}$ has finite kernel and cokernel by the snake lemma, so indeed
\begin{align*}
\rank_{\Z_p}(A_{K_i}^{\ast}) = \rank_{\Z_p}((A_{K_n}^{N_i})^{\ast}) = \rank_{\Z_p}((A_{K_n}^{\ast})_{N_i}) = \rank_{\Z_p}((A_{K_n}^{\ast})^{N_i}).
\end{align*}
Hence
{\allowdisplaybreaks
\begin{align*}
\sum_{i=0}^{n-1} \varphi(p^i)\lambda_{K_{n-i}} &= \sum_{i=0}^{n-1} \sum_{t=0}^{n-i} r_t \varphi(p^i)\varphi(p^t) \\
&= \sum_{i=0}^{n-1} \varphi(p^i)r_0 + \sum_{t=1}^n \sum_{i=0}^{n-t} \varphi(p^i)\varphi(p^t) r_t \\
&= \left(1+ (p-1)\sum_{j=0}^{n-2} p^j \right)r_0 + \sum_{t=1}^n r_t \varphi(p^t)\left(1 + (p-1)\sum_{j=0}^{n-t-1} p^j \right)  \\
&=  p^{n-1}r_0 + \varphi(p^n)(r_1+ \cdots +r_n) \\
&= p^{n-1}(np-n+1)r_0 + \varphi(p^n)(-nr_0 + r_1 + \cdots + r_n) \\
&=  p^{n-1}(np-n+1)\lambda_{K_0} + \varphi(p^n)\chi(G_n, A_{K_n})
\end{align*}}
which finishes the proof.
\end{proof}
\begin{corollary}\label{lazyname}
Let $p$ be prime and $K_0 \subseteq K_1 \subseteq \ldots \subseteq K_n$ be a tower of $\Z_p$-fields such that for all $i$ the extension $K_i/K_0$ is cyclic of degree $p^i$. Suppose $\mu_{K_0} = 0$. Then
\begin{align*} \lambda_{K_n} = p^n\lambda_{K_0} + \varphi(p^n)\chi(G_n, A_{K_n}) - (p-1) \sum_{i=1}^{n-1} \varphi(p^i)\chi(G_i, A_{K_i}) \end{align*}
where $G_i = \Gal(K_i/K_0)$.
\end{corollary}
\begin{proof}
It is clear that the statement holds when $n=0$ and also when $n=1$ by Iwasawa's formula, so we may assume $n \geq 2$ and that the formula is valid for $n-1$. Then by application of Theorem \ref{above} to $K_n/K_0$ and $K_{n-1}/K_0$ we get
\begin{align}\label{use}
\lambda_{K_n}-\lambda_{K_{n-1}} = \varphi(p^n)(\lambda_{K_0} + \chi(G_n,A_{K_n}) - \chi(G_{n-1},A_{K_{n-1}})),
\end{align}
and, in fact, this equation holds for $n=1$ as well. Induction yields
\begin{align*}
\lambda_{K_n} &= \varphi(p^n)(\lambda_{K_0} + \chi(G_n,A_{K_n}) - \chi(G_{n-1},A_{K_{n-1}})) \\
&{\hspace{0.2 in}} + p^{n-1}\lambda_{K_0} + \varphi(p^{n-1})\chi(G_{n-1}, A_{K_{n-1}}) - (p-1) \sum_{i=1}^{n-2} \varphi(p^i)\chi(G_i, A_{K_i}) \\
&= p^n\lambda_{K_0} + \varphi(p^n)\chi(G_n,A_{K_n}) - (p-1) \sum_{i=1}^{n-1} \varphi(p^i)\chi(G_i, A_{K_i})
\end{align*}
as needed.
\end{proof}
\begin{corollary}
\label{generalprop}
Let $p$ be prime and $K_0 \subseteq K_1 \subseteq \ldots \subseteq K_n$ be a tower of $\Z_p$-fields such that for all $i$ the extension $K_i/K_0$ is cyclic of degree $p^i$. Suppose $\mu_{K_0} = 0$. Then
\begin{align*}p^{n-1} \chi(G_n, A_{K_n}) = \sum_{i=1}^{n-1} \varphi(p^i)\chi(G_i, A_{K_i}) + \sum_{i=1}^n p^{n-i}\chi(N_{i-1}/N_i, A_{K_i}). \end{align*}
where $N_i = \Gal(K_n/K_i)$ and again $G_i = \Gal(K_i/K_0)$.
\end{corollary}
\begin{proof}
We have
\begin{align*}
p^{n-1}\chi(G_n, A_{K_n}) &=  \sum_{i=1}^{n-1} \varphi(p^i)\chi(G_i, A_{K_i}) + \frac{\lambda_{K_n} - p^n\lambda_{K_0}}{p-1} \\
&= \sum_{i=1}^{n-1} \varphi(p^i)\chi(G_i, A_{K_i}) + \sum_{i=1}^n p^{n-i}\chi(N_{i-1}/N_i, A_{K_i})
\end{align*}
where the first equality follows from Corollary \ref{lazyname} and the second equality follows from induction on Iwasawa's formula (\ref{Iwasform}).
\end{proof}
\begin{corollary}\label{gencong}
Let $p$ be prime and $K_0 \subseteq K_1 \subseteq \ldots \subseteq K_n$ be a tower of $\Z_p$-fields such that for all $i$ the extension $K_i/K_0$ is cyclic of degree $p^i$. As above, write $G_i = \Gal(K_i/K_0)$, $N_i = \Gal(K_n/K_i)$. Suppose $\mu_{K_0} = 0$. Then
\begin{enumerate}
\item for every $i=0, \ldots, n$
\begin{align*}\lambda_{K_n} \equiv \lambda_{K_i} \pmod{\varphi(p^{i+1})}\end{align*}
\item also,
\begin{align*}
-n\lambda_{K_0}\leq \chi(G_n, A_{K_n})
\end{align*}
\end{enumerate}
\end{corollary}
\begin{proof}
For part (1), we only need to note that Eq. (\ref{use}) in the proof of Corollary \ref{lazyname} implies 
\begin{align*}
\lambda_{K_n}\equiv \lambda_{K_{n-1}} \equiv \ldots \equiv \lambda_{K_i} \pmod{\varphi(p^{i+1})}
\end{align*}
 for all $i = 1, \ldots, n$.
To prove part (2), we note that
\begin{align*}
0 &\leq \sum_{i=0}^{n-1} \frac{\lambda_{K_{n-i}} - \lambda_{K_{n-i-1}}}{\varphi(p^{n-i})} = \frac{1}{\varphi(p^n)}\sum_{i=0}^{n-1} p^i(\lambda_{K_{n-i}} - \lambda_{K_{n-i-1}}) \\
&= \frac{1}{\varphi(p^n)}\left(\sum_{i=0}^{n-1} \varphi(p^i)\lambda_{K_{n-i}} - p^{n-1}\lambda_{K_{0}} \right) \\
&=  \frac{1}{\varphi(p^n)}\left(p^{n-1}(np-n+1)\lambda_{K_0} + \varphi(p^n)\chi(G_n, A_{K_n}) - p^{n-1}\lambda_{K_{0}} \right) \\
&= n\lambda_{K_0}+ \chi(G_n, A_{K_n}) \\
\end{align*}
again by Theorem \ref{above}.
\end{proof}
\begin{rem}
We can give a shorter, more direct proof of part (1) in Corollary \ref{gencong} above. Namely, we apply Lemma \ref{Rei} directly to get a short exact sequence
\begin{align*}
0 \rightarrow (A_{K_n}^{\ast})^{N_{n-1}} \hookrightarrow A_{K_n}^{\ast} \rightarrow \Z_p[\zeta_{p^n}]^{\oplus r} \rightarrow 0,
\end{align*}
so
\begin{align*}
\lambda_{K_n}  = \rank_{\Z_p}(A_{K_n}^{\ast}) = \rank_{\Z_p}((A_{K_n}^{\ast})^{N_{n-1}}) + \rank_{\Z_p}(\Z_p[\zeta_{p^n}]^{\oplus r}) = \lambda_{K_{n-1}} + r \varphi(p^n)
\end{align*}
as needed.
\end{rem}
Now we relate the $n$ Euler characteristics associated to subgroups (rather than quotients or subquotients)
\begin{align*}\chi(N_0,A_{K_n}), \chi(N_1, A_{K_n}), \ldots, \mbox{ and } \chi(N_{n-1}, A_{K_n})\end{align*}
to the two lambda invariants $\lambda_{K_n}$ and $\lambda_{K_0}$. The result is of a different nature since it involves non-integer coefficients.
\begin{thm}\label{subgroup}
Let $p$ be prime and $K_0 \subseteq K_1 \subseteq \ldots \subseteq K_n$ be a tower of $\Z_p$-fields such that for all $i$ the extension $K_i/K_0$ is cyclic of degree $p^i$. Suppose $\mu_{K_0} = 0$. Then $\mu_{K_1} = \cdots  = \mu_{K_n} = 0$ and
\begin{align*} \frac{\lambda_{K_n} - p^n\lambda_{K_0}}{p-1} = \frac{p^n\chi(N_0, A_{K_n})}{np-n+1} + \sum_{i=1}^{n-1} \frac{p^i(p-1)\chi(N_{n-i}, A_{K_n})}{(ip-i+p)(ip-i+1)}  \end{align*}
where $N_i = \Gal(K_n/K_i)$.
\end{thm}
\begin{proof}
We may assume $n\geq 2$ since the $n=0$ case is clear and the $n=1$ case is again a restatement of Iwasawa's formula. For each $i\geq 1$, we apply Theorem \ref{above} to the degree $p^i$ extension $K_n/K_{n-i}$ to find
\begin{align}\label{star}
\sum_{j=0}^{i-1} \varphi(p^j)\lambda_{K_{n-j}} &= p^{i-1}(ip-i+1)\lambda_{K_{n-i}} +\varphi(p^i)\chi(N_{n-i},A_{K_n}).
\end{align}
Now we multiply both sides of Eq. (\ref{star}) by $p((ip-i+p)(ip-i+1))^{-1}$ and sum over $i$, which gives
\begin{align*}
&\sum_{i=1}^{n-1}\frac{p^i\lambda_{K_{n-i}}}{ip-i+p} +\sum_{i=1}^{n-1}\frac{p^i(p-1)\chi(N_{n-i},A_{K_n})}{(ip-i+p)(ip-i+1)} \\
&= \sum_{i=1}^{n-1}\sum_{j=0}^{i-1}  \frac{p\varphi(p^j)\lambda_{K_{n-j}} }{(ip-i+p)(ip-i+1)} \\
&=  \sum_{j=0}^{n-2} \frac{p\varphi(p^j)\lambda_{K_{n-j}}}{p-1} \sum_{i=j+1}^{n-1} \left(\frac{1}{ip-i+1} - \frac{1}{(i+1)p-(i+1)+1} \right) \\
&= \sum_{j=0}^{n-2} \frac{p\varphi(p^j)\lambda_{K_{n-j}}}{p-1}\left(\frac{1}{jp-j+p} - \frac{1}{np-n+1} \right) \\
&= \frac{\lambda_{K_n}(np-n+1-p)}{(p-1)(np-n+1)} + \sum_{j=1}^{n-2} \frac{p^j\lambda_{K_{n-j}}}{jp-j+p} - \frac{1}{np-n+1}\sum_{j=1}^{n-2} p^j\lambda_{K_{n-j}}.
\end{align*}
Now we cancel like terms and again apply Eq. (\ref{star}) (with $i=n$) to conclude
\begin{align*}
&\sum_{i=1}^{n-1}\frac{p^i(p-1)\chi(N_{n-i},A_{K_n})}{(ip-i+p)(ip-i+1)} \\
&=  \frac{\lambda_{K_n}(np-n+1-p)}{(p-1)(np-n+1)} -\frac{p^{n-1}\lambda_{K_1}}{np-n+1} - \frac{1}{np-n+1}\sum_{j=1}^{n-2} p^j\lambda_{K_{n-j}} \\
&=  \frac{\lambda_{K_n}}{p-1} - \frac{p}{(p-1)(np-n+1)}\sum_{j=0}^{n-1} \varphi(p^j)\lambda_{K_{n-j}} \\
&= \frac{\lambda_{K_n}}{p-1}-\frac{p}{(p-1)(np-n+1)}(p^{n-1}(np-n+1)\lambda_{K_0} +\varphi(p^n)\chi(N_0,A_{K_n})) \\
&= \frac{\lambda_{K_n}-p^n\lambda_{K_0}}{p-1}-\frac{p^n\chi(N_0,A_{K_n})}{np-n+1}
\end{align*}
which completes the proof.
\end{proof}

\section{An Alternative Proof of Lemma \ref{work}}

Using a suggestion of Ralph Greenberg, we can use the structure thm for finitely generated $\Lambda$-modules to give a different proof of Lemma \ref{work}. Here $\Lambda = \Z_p[[T]]$ is an Iwasawa algebra.
\begin{thm}\label{lambdastructure}
Let $M$ be a finitely generated $\Lambda$-module. Then there is a $\Lambda$-module homomorphism
\begin{align*}
\theta\colon M \rightarrow \Lambda^r \oplus \bigoplus_{i=1}^s \frac{\Lambda}{(f_i(T)^{m_i})} \oplus \bigoplus_{j=1}^t \frac{\Lambda}{(p^{n_j})}
\end{align*}
such that $\ker(\theta), \coker(\theta)$ are finite and where each $f_i(T) \in \Z_p[T]$ is irreducible with $f_i(T) \equiv \mbox{power of }T \pmod{p}$.
\end{thm}
We will see that Lemma \ref{work} follows as an easy corollary of the following lemma.
\begin{lemma}\label{Lambda}
Let $G = \langle g \rangle \cong \Z/(p^n)$ for some prime $p$ and some nonnegative integer $n$. Suppose $M$ is a $\Z_pG$-module which is free of finite rank over $\Z_p$. There is an injective $\Z_pG$-module homomorphism with finite cokernel
\begin{align*}
M \rightarrowtail \bigoplus_{i=0}^n \Z_p[\zeta_{p^i}]^{\oplus r_i}
\end{align*}
for some nonnegative integers $r_0, \ldots, r_n$ where each $\Z_p[\zeta_{p^i}]$ has $\Z_pG$-module structure given by
\begin{align*}\Z_p[\zeta_{p^i}] \cong \frac{\Z_pG}{\Phi_{p^i}(g)\Z_pG}.\end{align*}
\end{lemma}
\begin{proof}
We know
\begin{align*}
\Lambda  \cong \varprojlim_{m \in \N} \Z_p[\Z/(p^m)] \colon T \mapsto (g_m - 1)_{m \in \N}
\end{align*}
with $\Z/(p^m)=\< g_m\>$ written multiplicatively, so $\Z_pG$ is a quotient ring of $\Lambda$. In this way, every $\Z_pG$-module is a $\Lambda$-module with $T$ acting as $g-1$, so Theorem \ref{lambdastructure} implies there is a $\Lambda$-module homomorphism
\begin{align*}
\theta \colon M \rightarrow \Z_p[[T]]^r \oplus \bigoplus_{i=1}^s \frac{\Z_p[[T]]}{(f_i(T)^{m_i})} \oplus \bigoplus_{j=1}^t \frac{\Z_p[[T]]}{(p^{n_j})}
\end{align*}
such that $\ker(\theta), \coker(\theta)$ are finite and where each $f_i(T) \in \Z_p[T]$ is irreducible with $f_i(T) \equiv \mbox{power of }T \pmod{p}$. Immediately, we see that $\ker(\theta) = 0$ and $r=t=0$ since $M$ is free and finitely generated over $\Z_p$. Since $(T+1)^{p^n}-1$ annihilates $M$, each $f_i(T)^{m_i}$ must be a divisor of $(T+1)^{p^n}-1 = \prod_{j=0}^n \Phi_{p^j}(T+1)$, so, in fact, $f_i(T)^{m_i} = \Phi_{p^j}(T+1)$ for some $j$ whence
\begin{align*}
 \frac{\Z_p[[T]]}{(f_i(T)^{m_i})} \cong  \frac{\Z_p[T]}{(\Phi_{p^j}(T+1))} \cong \frac{\Z_pG}{\Phi_{p^j}(g)\Z_pG}
\end{align*}
as $\Z_pG$-modules.
\end{proof}
\begin{rem}\label{regift}
Let $M,G = \<g\> \cong \Z/(p^n)$ be as in Lemma \ref{Lambda}. We can now give another proof of Lemma \ref{work}. Observe that if $C$ is a finite $\Z_pG$-module, then $\chi(N_i,C) = 0$ and $\rank_{\Z_p}(C^{N_i}) = 0$ for all $i \in \{0, \ldots, n\}$ where (as in Lemma \ref{work}) $N_i = \<g^{p^i}\>$. Thus since $\chi$ and $\rank_{\Z_p}$ are additive on short exact sequences, we see that it suffices to note the following computations:
\begin{align*}
\Z_p[\zeta_{p^j}]^{N_i} &= \left\{ \begin{array}{ll} \Z_p[\zeta_{p^j}] & \mbox{if } j \leq i \\ 0 & \mbox{if } j > i \end{array} \right. \\
\chi(N_i, \Z_p[\zeta_{p^j}]) &= \ord_p\left( \frac{|H^2(N_i,  \Z_p[\zeta_{p^j}])|}{|H^1(N_i,  \Z_p[\zeta_{p^j}])|} \right) \\
&= \left\{ \begin{array}{ll}  \ord_p\left|\frac{\Z_p[\zeta_{p^j}]}{p^{n-i}\Z_p[\zeta_{p^j}]} \right| = (n-i)\varphi(p^j) & \mbox{if } j \leq i \\ \ord_p\left|\frac{\Z_p[\zeta_{p^j}]}{(1-\zeta_{p^j}^{p^i})\Z_p[\zeta_{p^j}]} \right|^{-1} = -p^i & \mbox{if } j > i.  \end{array} \right.
\end{align*}
\end{rem}
\begin{rem}
The proof of Lemma \ref{Lambda} and Theorem \ref{above} show more than just formulas for Euler characteristics and lambda invariants. Indeed, they show a statement about representations. The proof is straightforward, so we shall forego it here.
\begin{thm}\label{genrep}
Let $K_0 \subseteq K_1 \subseteq \ldots \subseteq K_n$ be a tower of $\Z_p$-fields with $G_i = \Gal(K_i/K)$ and $N_i = \Gal(K_n/K_i) = \<g^{p^i}\> \cong \Z/(p^i)$ for all $i = 0, \ldots, n$. Assume $\mu_{K}=0$, define
\begin{align*}
V_{K_n} := A_{K_n}^{\ast} \otimes_{\Z_p} \Q_p,
\end{align*}
and let $\pi_{K_n/K_0}$ be the corresponding representation. Then we have an isomorphism of $\Q_p$-representations (with the appropriate interpretation for negative coefficients)
\begin{align*}
\pi_{K_n/K_0} \cong \lambda_K\pi_{G_n} \oplus  \bigoplus_{i=1}^n \left( \chi(G_i, A_{K_i}) - \chi(G_{i-1}, A_{K_{i-1}}) \right)\pi_{\varphi(p^i)}
\end{align*}
where $\pi_{G_n}$ is the regular representation and $\pi_d$ is the unique faithful, irreducible representation of degree $d \in \{\varphi(p), \varphi(p^2), \ldots, \varphi(p^n)\}$. Comparing the degrees of both sides of the isomorphism recovers Corollary \ref{lazyname}.
\end{thm}
\end{rem}

\section{Vanishing Criteria for Iwasawa Lambda Invariants}

In this section we give a couple of generalized vanishing criteria for Iwasawa lambda invariants. The criteria will apply to certain cyclic extensions of $\Z_p$-fields of degree $p^n$ and will generalize the results found in \cite{Fuku} of Fukuda et al. We need a couple of lemmas.
\begin{lemma}\label{useful}
Let $L/K$ be a cyclic $p$-extension of $\Z_p$-fields with $G = \Gal(L/K)$. Suppose $\mu_K=\lambda_K = 0$. Then
\begin{align*}
\ord_p|H^1(G,\OO_L^{\times})| + \ord_p|(I_L^GP_L)/(I_KP_L) | = \chi(G,I_L)
\end{align*}
\end{lemma}
\begin{proof}
There is a short exact sequence of $\Z_pG$-modules
\begin{align*}
(I_KP_L^G)/I_K \rightarrowtail I_L^G/I_K \twoheadrightarrow I_L^G/(I_KP_L^G).
\end{align*}
Also,  
\begin{align*}
(I_K \cap P_L^G)/P_K \subseteq P_L^G/P_K \cong H^1(G,\OO_L^{\times})
\end{align*}
is a $p$-group, so
\begin{align*}
(I_K \cap P_L^G)/P_K \subseteq A_K \cong 0
\end{align*}
by our $\mu_K = \lambda_K=0$ assumption. Thus $I_K \cap P_L^G = P_K$, so
\begin{align*}
\frac{I_KP_L^G}{I_K} &\cong \frac{P_L^G}{I_K \cap P_L^G} = \frac{P_L^G}{P_K} \cong H^1(G,\OO_L^{\times})
\end{align*}
and
\begin{align*}
\frac{I_L^G}{I_KP_L^G} &= \frac{I_L^G}{I_L^G \cap (I_KP_L)} \cong \frac{I_L^GP_L}{I_KP_L}.
\end{align*}
This completes the proof since
\begin{align*}
\ord_p|I_L^G/I_K| = \chi(G,I_L)
\end{align*}
by Lemma \ref{chi}.
\end{proof}
Now we can state and prove the first vanishing criterion.
\begin{thm}\label{vanishing1}
Let $L/K$ be a cyclic $p$-extension of $\Z_p$-fields which is unramified at every infinite place with $G = \Gal(L/K)$. Suppose $\mu_K = 0$. Then $\lambda_L=0$ if and only if the following three conditions hold:
\begin{enumerate}[(i)]
\item $\lambda_K=0$
\item $\ord_p|H^2(G,\OO_L^{\times})| = 0$
\item $\ord_p|(I_L^GP_L)/(I_KP_L)|=0$
\end{enumerate}
\end{thm}
\begin{proof}
Condition (1) is obviously necessary for $\lambda_L=0$, so we will assume $\lambda_K=0$ and must now show that $\lambda_L=0 \Leftrightarrow$ conditions (2) and (3) hold. Consider the tower
\begin{align*}
K=K_0 \subseteq K_1 \subseteq \ldots \subseteq K_n = L
\end{align*}
of $\Z_p$-fields where $G_i = \Gal(K_i/K) \cong \Z/(p^i)$ for all $i = 0, \ldots, n$. Our $\lambda_K = 0$ assumption ensures that $\chi(G_i, A_{K_i}) \geq 0$ for all $i = 1, \ldots, n$ by part (2) of Corollary \ref{gencong}. Thus if $\chi(G,A_L)=0$, then Corollary \ref{lazyname} implies
\begin{align*} 0 \leq (p-1) \sum_{i=1}^{n-1} \varphi(p^i)\chi(G_i, A_{K_i}) = \varphi(p^n)\chi(G, A_L) - \lambda_L = -\lambda_L \leq 0, \end{align*}
so here $\lambda_L= 0$. Conversely, if $\lambda_L=0$, then $\lambda_{K_i}=0$ for all $i=0, \ldots, n$, so $\chi(G,A_L)=0$ by Theorem \ref{above}. Therefore $\lambda_L= 0 \Leftrightarrow \chi(G,A_L) = 0$, but Lemma \ref{useful} and Lemma \ref{chi} imply
\begin{align*}
\chi(G, A_L) &= \ord_p|H^2(G, \OO_L^{\times})| - \ord_p|H^1(G,\OO_L^{\times})|+\chi(G,I_L) \\
&=  \ord_p|H^2(G, \OO_L^{\times})| +  \ord_p|(I_L^GP_L)/(I_KP_L)| \geq 0,
\end{align*}
so $\lambda_L = 0 \Leftrightarrow \ord_p|H^2(G, \OO_L^{\times})| =  \ord_p|(I_L^{G}P_L)/(I_KP_L)| =0$.
\end{proof}
To derive our next lemma and establish the second vanishing criterion, we state the following result. A proof can be found, for example, in \cite{gree}.
\begin{thm}\label{Green}
Let ${l}/k$ be a Galois extension of number fields with $G = \Gal({l}/k)$. Then there is an exact sequence of abelian groups
\begin{align*}0 &\rightarrow \ker(J_{{l}/k}) \rightarrow H^1(G, \OO_{{l}}^{\times}) \rightarrow \bigoplus_{v} \frac{\Z}{(e(w/v))} \rightarrow C_{{l}}^{[G]}/J_{{l}/k}(C_k) \rightarrow 0 \end{align*}
where $C_{{l}}^{[G]}$ is the subgroup of $C_{{l}}^G$ generated by classes of $G$-fixed ideals, the direct sum ranges over all finite places $v$ of $k$ having ramification index $e(w/v)$ with $w$ a place of ${l}$ lying over $v$, and $J_{{l}/k}\colon C_k \rightarrow C_{{l}}$ is the natural map sending the class $[I]$ of an ideal $I$ to the class $[\OO_{{l}}I]$. Further, if $G$ is cyclic and ${l}/k$ is unramified at every infinite place, then
\begin{align*}
\frac{|H^2(G,\OO_{{l}}^{\times})|}{|H^1(G,\OO_{{l}}^{\times})|} = \frac{1}{[{l}:k]}.
\end{align*}
\end{thm}
\begin{lemma}\label{H2=0}
Let $L/K$ be a cyclic $p$-extension of $\Z_p$-fields which is unramified at every infinite place. Suppose $K = k\Q_{\infty}$ is the cyclotomic $\Z_p$-extension of a number field $k$ such that $p$ does not divide the class number of $k$ and $k$ has only one prime lying above $p$. Then
\begin{align*}\ord_p|H^2(G, \OO_L^{\times})| =0.\end{align*}
where $G=\Gal(L/K)$.
\end{lemma}
\begin{proof}
Here we generalize the method of proof found in \cite{Fuku}, where the result is proved in the case that $L$ is totally real and $[L:K]=p$. First, note that if $\p$ is the unique prime ideal of $k$ lying over $p$, then $\p_n/\p$ is totally ramified in $k_n/k$ and $p$ does not divide the class number of $k_n$ for all nonnegative integers $n$. Thus using Theorem \ref{Green} on the extension $k_n/k_m$ we find that for all nonnegative integers $m,n$ with $m\leq n$
\begin{align*}
\ord_p|\OO_{k_m}^{\times}/N_{k_n/k_m}(\OO_{k_n}^{\times})| &= \ord_p|H^2(\Gal(k_n/k_m), \OO_{k_n}^{\times})| \\
&= -(n-m)+ \ord_p|H^1( \Gal(k_n/k_m), \OO_{k_n}^{\times})| \\
&\leq -(n-m) +\ord_p(e(\p_n/\p_m)) +\ord_p|\ker(J_{k_n/k_m})| \\
&\leq -(n-m) + (n-m) + \ord_p|C_{k_m}| =0.
\end{align*}
Thus $N_{k_n/k_m}(\OO_{k_n}^{\times}) =\OO_{k_m}^{\times}$ for all nonnegative integers $m,n$ with $m \leq n$, so if $L=l\Q_{\infty}$ for some number field ${l}$ with $\Gal({l}/k)\cong \Gal(L/K) \cong \Z/(p^d)$, then the induced maps
\begin{align*}
\widetilde{N}_{k_n/k_m} \colon \frac{\OO_{k_n}^{\times}}{N_{{l}_n/k_n}(\OO_{{l}_n}^{\times})} \longrightarrow \frac{\OO_{k_m}^{\times}}{N_{{l}_m/k_m}(\OO_{{l}_m}^{\times})}
\end{align*}
are surjective for all nonnegative integers $m,n$ with $m\leq n$. On the other hand, letting $s_n$ be the number of ramified primes of $k_n$ in ${l}_n/k_n$ and $s_{\infty} < \infty$ be the number of ramified primes of $K$ in $L/K$, we see that Theorem \ref{Green} applied to the extension ${l}_n/k_n$ with $\Gal({l}_n/k_n) \cong \Gal(L/K)\cong \Z/(p^d)$ gives
\begin{align*}
\ord_p|\OO_{k_n}^{\times}/N_{{l}_n/k_n}(\OO_{{l}_n}^{\times})| &= \ord_p|H^2(\Gal({l}_n/k_n), \OO_{{l}_n}^{\times})| \\
&= -d + \ord_p|H^1(\Gal({l}_n/k_n), \OO_{{l}_n}^{\times})| \\
&\leq -d+ \sum_{i=1}^{s_n} \ord_p(e(w_i/v_i)) +\ord_p|\ker(J_{l_n/k_n})|  \\
&\leq -d + ds_n + \ord_p|C_{k_n}|\\
&\leq d(s_\infty - 1).
\end{align*}
Therefore the maps $\widetilde{N}_{k_n/k_m}$ are isomorphisms of finite abelian groups for sufficiently large $m,n$. Now consider the canonical maps
\begin{align*}
\widetilde{\rho}_{k_n/k_m} \colon \frac{\OO_{k_m}^{\times}}{N_{{l}_m/k_m}(\OO_{{l}_m}^{\times})} \longrightarrow \frac{\OO_{k_n}^{\times}}{N_{{l}_n/k_n}(\OO_{{l}_n}^{\times})}
\end{align*}
for $m\leq n$. These maps have the property that $\widetilde{N}_{k_n/k_m} \circ \widetilde{\rho}_{k_n/k_m}$ is the exponentiation by $p^{n-m}$ map when the groups are written multiplicatively. Thus when $n-m\geq d(s_{\infty} - 1)$ the composition $\widetilde{N}_{k_n/k_m} \circ \widetilde{\rho}_{k_n/k_m}$ is the trivial map, but $\widetilde{N}_{k_n/k_m}$ is an isomorphism for sufficiently large $m$, so $\widetilde{\rho}_{k_n/k_m}$ is the trivial map when $m$ is sufficiently large and $n \geq m+ d(s_{\infty} - 1)$. Therefore
\begin{align*}
H^2(G, \OO_L^{\times}) \cong \varinjlim_n H^2(G_n, \OO_{{l}_n}^{\times}) \cong 0
\end{align*}
which finishes the proof.
\end{proof}
Now we are ready to give the more specialized and easily applicable vanishing criterion.
\begin{thm}\label{vanishing2}
Let $L/K$ be a cyclic $p$-extension of $\Z_p$-fields which is unramified at every infinite place. Suppose $K = k\Q_{\infty}$ is the cyclotomic $\Z_p$-extension of a number field $k$ such that $p$ does not divide the class number of $k$ and $k$ has only one prime lying above $p$. Then $\lambda_L=0$ if and only if, for all prime ideals $\p$ of $K$ which ramify in $L/K$ and do not lie over $p$, the order in $C_L$ of the class of the product of prime ideals of $L$ lying over $\p$ is prime to $p$.
\end{thm}
\begin{proof}
The ``$\Rightarrow$'' implication is clear. The ``$\Leftarrow$'' implication follows from Theorem \ref{vanishing1} since the assumptions we made ensure that conditions (1) and (2) hold by Iwasawa's well-known vanishing criterion and Lemma \ref{H2=0}, respectively, and with regard to condition (3), we only need to note that $(I_L^GP_L)/(I_KP_L)$ is a $p$-group generated by the classes of products of prime ideals of $L$ lying over $\p$ where $\p$ runs through all prime ideals of $K$ which ramify in $L/K$ and do not lie above $p$.
\end{proof}

\section*{Acknowledgements}

This work is adapted from a portion of the author's Ph.D. thesis at the University of Arizona \cite{thesis}. He thanks his advisor, William McCallum, for his time and encouragement. The author would also like to acknowledge the many helpful comments of the reviewers including the correction of a few typos/errors and the streamlining of several proofs. In particular, the proofs of Theorems \ref{subgroup} and \ref{vanishing1} were greatly improved by their suggestions.

\end{document}